\documentclass[12pt,a4paper,oneside]{amsart}
\usepackage{fullpage}
\usepackage{amssymb,amsfonts}
\usepackage{epsfig}
\usepackage{graphicx}

\usepackage{graphics}
\newtheorem{theorem}{Theorem}[section]
\newtheorem{lemma}[theorem]{Lemma}

\theoremstyle{definition}
\newtheorem{definition}[theorem]{Definition}
\newtheorem{remark}[theorem]{Remark}

\numberwithin{equation}{section}

\DeclareMathOperator{\diam}{diam} 

\newcommand{\be}{\begin{equation}}
\newcommand{\ee}{\end{equation}}

\newcommand{\dist}{{\operatorname{dist}}}

\DeclareMathOperator*{\essinf}{ess\,inf}

\makeindex

\def\Xint#1{\mathchoice 
 {\XXint\displaystyle\textstyle{#1}}%
{\XXint\textstyle\scriptstyle{#1}}%
{\XXint\scriptstyle\scriptscriptstyle{#1}}%
 {\XXint\scriptscriptstyle\scriptscriptstyle{#1}}%
 \!\int}
\def\XXint#1#2#3{{\setbox0=\hbox{$#1{#2#3}{\int}$}
 \vcenter{\hbox{$#2#3$}}\kern-.5\wd0}}

 \def\dashint{\Xint-}

\begin{document}
\title{Measure density and Embeddings of Haj\l asz-Besov and Haj\l asz-Triebel-Lizorkin spaces}
\author{Nijjwal Karak}
\address{Discipline of Mathematics, Indian Institute of Technology Indore, Simrol, Indore-453552, India}
\email{nijjwal@gmail.com}
\thanks{I would like to thank Professor Pekka Koskela for introducing me with the problem and for his fruitful suggestions. This work was supported by DST-SERB (Grant no. PDF/2016/000328).}
\begin{abstract}
In this paper, we investigate the relation between Sobolev-type embeddings of Haj\l asz-Besov spaces (and also Haj\l asz-Triebel-Lizorkin spaces) defined on a metric measure space $(X,d,\mu)$ and lower bound for the measure $\mu.$ We prove that if the measure $\mu$ satisfies $\mu(B(x,r))\geq cr^Q$ for some $Q>0$ and for any ball $B(x,r)\subset X,$ then the Sobolev-type embeddings hold on balls for both these spaces. On the other hand, if the Sobolev-type embeddings hold in a domain $\Omega\subset X,$ then we prove that the domain $\Omega$ satisfies the so-called measure density condition, i.e., $\mu(B(x,r)\cap\Omega)\geq cr^Q$ holds for any ball $B(x,r)\subset X,$ where $X=(X,d,\mu)$ is an Ahlfors $Q$-regular and geodesic metric measure space. 
\end{abstract}
\maketitle
\section{Introduction}
The most important result of the classical theory of Sobolev spaces is the Sobolev embedding theorem. Embeddings of fractional Sobolev spaces $W^{s,p}(\Omega),$ where $\Omega$ is a domain in $\mathbb{R}^n$ and $0<s<1,$ have been established in \cite{EGE12} when $p\geq 1$ and in \cite{Zho15} when $p<1.$ In the metric space setting, especially for Haj\l asz-Sobolev space $M^{1,p}(X),$ Haj\l asz has been able to find similar embeddings on balls provided that the measure of the balls has a lower bound, see Theorem 8.7 of \cite{Haj03}. We assume here and throughout the paper that $X=(X,d,\mu)$ is a \textit{metric measure space} equipped with a metric $d$ and a Borel regular measure $\mu$ on $X$ such that all balls defined by $d$ have finite and positive measures. In this paper we have proved similar embeddings on balls for homogeneous Haj\l asz-Besov spaces $\dot{N}^s_{p,q}(X)$ and also for homogeneous Haj\l asz-Triebel-Lizorkin spaces $\dot{M}^s_{p,q}(X),$ see section 3 and 4. For the definitions of $M^{s,p}(X),$ $\dot{M^{s,p}(X)},$ $M^s_{p,q}(X),$ $\dot{M}^s_{p,q}(X),$  $N^s_{p,q}(X)$ and $\dot{N}^s_{p,q}(X)$ see Section 2. Among several possible definitions of Besov and Triebel-Lizorkin spaces in the metric setting, the pointwise definition introduced in \cite{KYZ11} appears to be very useful. This approach is based on the definition of Haj\l asz-Sobolev space; it leads to the classical Besov and Triebel-Lizorkin spaces in the setting of Euclidean space, \cite[Theorem 1.2]{KYZ11} and it gives a simple way to define these spaces on a measurable subset of $\mathbb{R}^n.$
\begin{definition}
Let $(X,d)$ be a metric space equipped with a measure $\mu.$ A measurable set $S\subset X$ is said to satisfy a \textit{measure density condition}, if there exists a constant $c_m>0$ such that
\begin{equation}\label{measuredensity}
\mu(B(x,r)\cap S)\geq c_m\mu(B(x,r))
\end{equation}
for all balls $B(x,r)$ with $x\in S$ and $0<r\leq 1.$
\end{definition}
\noindent Note that sets satisfying such a condition are sometimes called in the literature regular sets. If the measure $\mu$ is doubling, then the upper bound 1 for the radius $r$ can be omitted. If a set $S$ satisfies the measure density condition, then we have $\mu(\overline{S}\setminus S)=0,$ \cite[Lemma 2.1]{Shv07}. Some examples of sets satisfying the measure density condition are cantor-like sets such as Sierpi\'nski carpets of positive measure.\\
\indent In \cite[Theorem 1]{HKT08b}, the authors have proved that if the Sobolev embedding holds in a domain $\Omega\subset\mathbb{R}^n,$ in any of all the possible cases, then $\Omega$ satisfies the measure density condition. Same result for fractional Sobolev spaces was obtained by Zhou, \cite{Zho15}. In this paper, we have obtained similar results for Haj\l asz-Besov spaces $N^s_{p,q}$ and Haj\l asz-Triebel-Lizorkin spaces $M^s_{p,q},$ see section 5. The idea of the proof is borrowed from \cite[Theorem 6.1]{HIT16} where the authors showed that an $M^s_{p,q}$-extension domain (or an $N^s_{p,q}$-extension domain) satisfies measure density condition. \\
\indent See \cite{HHL} for geometric characterizations of embedding theorems for these spaces.\\
\indent Notation used in this paper is standard. The symbol $c$ or $C$ will be used to designate a general constant which is independent of the main parameters and whose value may change even within a single string of estimate. The symbol $A\lesssim B$ or $B \gtrsim A$ means that $A\leq CB$ for some constant $C.$ If $A\lesssim B$ and $B\lesssim A,$ then we write $A\approx B.$ For any locally integrable function $u$ and $\mu$-measurable set $A,$ we denote by $\dashint_{A}u$ the integral average of $u$ on A, namely, $\dashint_{A}u:=\frac{1}{\mu(A)}\int_Au.$  
\section{Definitions and Preliminaries}
Besov and Triebel-Lizorkin spaces are certain generalizations of fractional Sobolev spaces. There are several ways to define these spaces in the Euclidean setting and also in the metric setting. For various definitions of in the metric setting, see \cite{GKS10}, \cite{GKZ13}, \cite{KYZ11} and the references therein. In this paper, we use the approach based on pointwise inequalities, introduced in \cite{KYZ11}.
\begin{definition}
Let $S\subset X$ be a measurable set and let $0<s<\infty.$ A sequence of nonnegative measurbale functions $(g_k)_{k\in\mathbb{Z}}$ is a fractional $s$-gradient of a measurable function $u:S\rightarrow [-\infty,\infty]$ in $S,$ if there exists a set $E\subset S$ with $\mu(E)=0$ such that
\begin{equation}\label{Hajlasz}
\vert u(x)-u(y)\vert\leq d(x,y)^s\left(g_k(x)+g_k(y)\right)
\end{equation} 
for all $k\in\mathbb{Z}$ and for all $x,y\in S\setminus E$ satisfying $2^{k-1}\leq d(x,y)<2^{k}.$ The collection of all fractional $s$-gradient of $u$ is denoted by $\mathbb{D}^s(u).$
\end{definition}
Let $S\subset X$ be a measurable set. For $0<p,q\leq \infty$ and a sequence $\vec{f}=(f_k)_{k\in\mathbb{Z}}$ of measurable functions, we define
\begin{equation*}
\Vert (f_k)_{k\in\mathbb{Z}}\Vert_{L^p(S,l^q)}=\big\Vert \Vert (f_k)_{k\in\mathbb{Z}}\Vert_{l^q} \big\Vert_{L^p(S)}
\end{equation*}
and
\begin{equation*}
\Vert (f_k)_{k\in\mathbb{Z}}\Vert_{l^q(L^p(S))}=\big\Vert (\Vert (f_k)\Vert_{L^p(S)})_{k\in\mathbb{Z}} \big\Vert_{l^q},
\end{equation*}
where
\begin{equation*}
\Vert (f_k)_{k\in\mathbb{Z}}\Vert_{l^q}=
\begin{cases}
(\sum_{k\in\mathbb{Z}}\vert f_k\vert^q)^{1/q},& ~\text{when}~0<q<\infty,\\
\sup_{k\in\mathbb{Z}}\vert f_k\vert,& ~\text{when}~q=\infty.
\end{cases}
\end{equation*}
\begin{definition}
Let $S\subset X$ be a measurable set. Let $0<s<\infty$ and $0<p,q\leq\infty.$ The \textit{homogeneous Haj\l asz-Triebel-Lizorkin space} $\dot{M}^s_{p,q}(S)$ consists of measurable functions $u:S\rightarrow [-\infty,\infty],$ for which the (semi)norm
\begin{equation*}
\Vert u\Vert_{\dot{M}^s_{p,q}(S)}=\inf_{\vec{g}\in\mathbb{D}^s(u)}\Vert\vec{g}\Vert_{L^p(S,l^q)}
\end{equation*}
is finite. The (non-homogeneous) \textit{Haj\l asz-Triebel-Lizorkin space} $M^s_{p,q}(S)$ is $\dot{M}^s_{p,q}(S)\cap L^p(S)$ equipped with the norm
\begin{equation*}
\Vert u\Vert_{M^s_{p,q}(S)}=\Vert u\Vert_{L^p(S)}+\Vert u\Vert_{\dot{M}^s_{p,q}(S)}.
\end{equation*} 
Similarly, the \textit{homogeneous Haj\l asz-Besov space} $\dot{N}^s_{p,q}(S)$ consists of measurable functions $u:S\rightarrow [-\infty,\infty],$ for which the (semi)norm
\begin{equation*}
\Vert u\Vert_{\dot{N}^s_{p,q}(S)}=\inf_{\vec{g}\in\mathbb{D}^s(u)}\Vert\vec{g}\Vert_{l^q(L^p(S))}
\end{equation*}
is finite and the (non-homogeneous) \textit{Haj\l asz-Besov space} $N^s_{p,q}(S)$ is $\dot{N}^s_{p,q}(S)\cap L^p(S)$ equipped with the norm
\begin{equation*}
\Vert u\Vert_{N^s_{p,q}(S)}=\Vert u\Vert_{L^p(S)}+\Vert u\Vert_{\dot{N}^s_{p,q}(S)}.
\end{equation*} 
\end{definition}
\noindent The space $M^s_{p,q}(\mathbb{R}^n)$ given by the metric definition coincides with the Triebel-Lizorkin space $F^s_{p,q}(\mathbb{R}^n),$ defined via the Fourier analytic approach, when $0<s<1,$ $n/(n+s)<p<\infty$ and $0<q\leq \infty,$ see \cite{KYZ11}. Similarly, $N^s_{p,q}(\mathbb{R}^n)$ coincides with Besov space $B^s_{p,q}(\mathbb{R}^n)$ for $0<s<1,$ $n/(n+s)<p<\infty$ and $0<q\leq \infty,$ see \cite{KYZ11}. For the definitions of $F^s_{p,q}(\mathbb{R}^n)$ and $B^s_{p,q}(\mathbb{R}^n),$ we refer to \cite{Tri83} and \cite{Tri92}.
\begin{definition}
Let $S\subset X$ be a measurable set. Let $0<s<\infty$ and $0<p\leq\infty.$ A nonnegative measurable function $g$ is an $s$-gradient of a measurable function $u$ in $S$ if there exists a set $E\subset S$ with $\mu(E)=0$ such that for all $x,y\in S\setminus E,$
\begin{equation*}
\vert u(x)-u(y)\vert\leq d(x,y)^s(g(x)+g(y)).
\end{equation*} 
The collection of all $s$-gradients of $u$ is denoted by $\mathcal{D}^s(u).$ The \textit{homogeneous Haj\l asz-Sobolev space} $\dot{M}^{s,p}(S)$ consists of measurable functions $u$ for which 
\begin{equation*}
\Vert u\Vert_{\dot{M}^{s,p}(S)}=\inf_{g\in\mathcal{D}^s(u)}\Vert g\Vert_{L^p(S)}
\end{equation*}
is finite. The \textit{Haj\l asz-Sobolev space} $M^{s,p}(S)$ is $\dot{M}^{s,p}(S)\cap L^p(S)$ equipped with the norm
\begin{equation*}
\Vert u\Vert_{M^{s,p}(S)}=\Vert u\Vert_{L^p(S)}+\Vert u\Vert_{\dot{M}^{s,p}(S)}.
\end{equation*}
\end{definition}
\noindent Note that if $0<s<\infty$ and $0<p\leq\infty,$ then $\dot{M}^s_{p,\infty}(X)=\dot{M}^{s,p}(X),$ \cite[Proposition 2.1]{KYZ11}.\\
\indent Let $(X,d,\mu)$ be a metric measure space. The measure $\mu$ is called \textit{doubling} if there exist a constant $C_{\mu}\geq 1$ such that
\begin{equation*}
\mu(B(x,2r))\leq C_{\mu}\,\mu(B(x,r))
\end{equation*}
for each $x\in X$ and $r>0.$ We call a triple $(X,d,\mu)$ a \textit{doubling metric measure space} if $\mu$ is a doubling measure on $X.$\\
\indent As a special case of doubling spaces we consider $Q$-regular spaces. The space $X$ is said to be $Q$-regular, $Q>1,$ if there is a constant $c_Q\geq 1$ such that
\begin{equation*}
c_Q^{-1}r^Q\leq \mu(B(x,r))\leq c_Qr^Q
\end{equation*} 
for each $x\in X$ and for all $0<r\leq\diam X.$\\
\indent A metric space $X$ is said to be \textit{geodesic} if every pair of points in the space can be joined by a curve whose length is equal to the distance between the points.\\
\indent We will often use the following elementary inequality, which holds whenever $a_i\geq 0$ for all $i$ and $0<\beta\leq 1,$
\begin{equation}\label{inequality}
\sum_{i\in\mathbb{Z}}a_i\leq\Big(\sum_{i\in\mathbb{Z}}a_i^{\beta}\Big)^{1/\beta}.
\end{equation}
\section{Haj\l asz-Triebel-Lizorkin spaces}
We use the idea of Haj\l asz from \cite{Haj03} to prove the following theorem. We will skip the case $q=\infty$ as it is proved in \cite[Thorem 8.7]{Haj03} when $s=1$ and other cases can be derived by modifying the proof of it.
\begin{theorem}\label{embedding}
Let $(X,d,\mu)$ be a metric measure space and $B_0$ be a fixed ball of radius $r_0.$ Let us assume that the measure $\mu$ has a lower bound, that is there exist constants $b, Q>0$ such that $\mu(B(x,r))\geq br^Q$ whenever $B(x,r)\subset 2B_0.$ Let $u\in \dot{M}^s_{p,q}(2B_0)$ and $\vec{g}=(g_j)\in \mathbb{D}^s(u)$ where $0<p,q,s<\infty.$  Then there exist constants $C,\,C_1,\,C_2$ and $C_3$ such that\\
$1.$ If $0<sp<Q,$ then $u\in L^{p^*}(B_0),$ $p^*=\frac{Qp}{Q-sp}$ and
\begin{equation}\label{embed}
\inf_{c\in\mathbb{R}}\left(\dashint_{B_0}\vert u-c\vert^{p^*}\,d\mu\right)^{\frac{1}{p^*}}\leq C\left(\frac{\mu(2B_0)}{br_0^Q}\right)^{1/p}r_0^{s}\left(\dashint_{2B_0}\bigg(\sum_{j=-\infty}^{\infty}g_j^q\bigg)^{\frac{p}{q}}\,d\mu\right)^{\frac{1}{p}}.
\end{equation}
$2.$ If $sp=Q,$ then
\begin{equation}\label{embedb}
\dashint_{B_0}\exp\left(C_1b^{1/Q}\frac{\vert u-u_{B_0}\vert}{\Vert \vec{g}\Vert_{L^p(2B_0,l^q)}}\right)\,d\mu\leq C_2.
\end{equation}
$3.$ If $sp>Q,$ then
\begin{equation}\label{embedc}
\Vert u-u_{B_0}\Vert_{L^{\infty}(B_0)}\leq C_3\left(\frac{\mu(2B_0)}{br_0^Q}\right)^{1/p}r_0^{s}\left(\dashint_{2B_0}\bigg(\sum_{j=-\infty}^{\infty}g_j^q\bigg)^{\frac{p}{q}}\,d\mu\right)^{\frac{1}{p}}.
\end{equation}
In particular, for $x,y\in B_0,$ we have
\begin{equation}\label{embedc'}
\vert u(x)-u(y)\vert\leq cb^{-1/p}d(x,y)^{1-s/p}\left(\dashint_{2B_0}\bigg(\sum_{j=-\infty}^{\infty}g_j^q\bigg)^{\frac{p}{q}}\,d\mu\right)^{\frac{1}{p}}.
\end{equation}
\end{theorem}
\begin{proof}
We may assume by selecting an appropriate constant that $\essinf_E u=0,$ where $E\subset 2B_0$ is any subset of positive measure, since subtracting a constant from $u$ will not affect the inequality \eqref{embed}. The set $E$ will be chosen later. With a correct choice of $E$ we will prove \eqref{embed} with $(\dashint_{B_0}\vert u\vert ^{p^*}\,d\mu)^{1/p^*}$ on the left hand side.\\
\indent If $\sum_{j=-\infty}^{\infty}g_j^q=0$ a.e., then $g_j=0$ a.e. for all $j,$ which implies that $u$ is constant and hence the theorem follows. Thus we may assume that $\int_{2B_0}(\sum_j g_j^q)^{\frac{p}{q}}\,d\mu>0$. We may also assume that
\begin{equation}\label{lowerbound}
\bigg(\sum_{j=-\infty}^{\infty}g_j(x)^q\bigg)^{\frac{1}{q}}\geq 2^{-(1+\frac{1}{p})}\left(\dashint_{2B_0}\bigg(\sum_{j=-\infty}^{\infty}g_j(x)^q\bigg)^{\frac{p}{q}}\,d\mu\right)^{\frac{1}{p}}>0
\end{equation}
for all $x\in 2B_0$ as otherwise we can replace $\left(\sum g_j(x)^q\right)^{1/q}$ by 
\begin{equation*}
\left(\sum \widetilde{g}_j(x)^q\right)^{1/q}=\left(\sum g_j(x)^q\right)^{1/q}+\left(\dashint_{2B_0}\bigg(\sum g_j(x)^q\bigg)^{\frac{p}{q}}\,d\mu\right)^{\frac{1}{p}}.
\end{equation*}
Let us define auxiliary sets
\begin{equation*}
E_k=\bigg\{x\in 2B_0:\Big(\sum_{j=-\infty}^{\infty}g_j(x)^q\Big)^{\frac{1}{q}}\leq 2^k\bigg\},\quad k\in\mathbb{Z}.
\end{equation*}
Clearly $E_k\subset E_{k+1}$ for all $k.$ Observe that
\begin{equation}\label{rhs}
\int_{2B_0}\Big(\sum_j g_j^q\Big)^{\frac{p}{q}}\,d\mu\approx\sum_{k=-\infty}^{\infty}2^{kp}\mu(E_k\setminus E_{k-1}).
\end{equation}
Let $a_k=\sup_{B_0\cap E_k}\vert u\vert.$ Obviously, $a_k\leq a_{k+1}$ and
\begin{equation}\label{lhs}
\int_{B_0}\vert u\vert^{p^*}\,d\mu\leq\sum_{k=-\infty}^{\infty}a_k^{p^*}\mu(B_0\cap (E_k\setminus E_{k-1})).
\end{equation}
By Chebyschev's inequality, we get an upper bound of the complement of $E_k$
\begin{eqnarray}\label{Chebyschev}
\mu(2B_0\setminus E_k) &=& \mu\bigg(\Big\{x\in 2B_0:\Big(\sum_{j=-\infty}^{\infty}g_j(x)^q\Big)^{\frac{1}{q}}>2^k\Big\}\bigg) \nonumber \\
&\leq & 2^{-kp}\int_{2B_0}\Big(\sum_j g_j^q\Big)^{\frac{p}{q}}\,d\mu .
\end{eqnarray}
Lower bound \eqref{lowerbound} implies that $E_k=\emptyset$ for sufficiently small $k.$ On the other hand $\mu(E_k)\rightarrow \mu(2B_0)$ as $k\rightarrow\infty.$ Hence there is $\widetilde{k}_0\in\mathbb{Z}$ such that
\begin{equation}\label{conv}
\mu(E_{\widetilde{k}_0-1})<\frac{\mu(2B_0)}{2}\leq\mu(E_{\widetilde{k}_0}).
\end{equation}
The inequality on the right hand side gives $E_{\widetilde{k}_0}=\neq\emptyset$ and hence according to \eqref{lowerbound}
\begin{equation}\label{first}
2^{-(1+\frac{1}{p})}\left(\dashint_{2B_0}\Big(\sum_{j=-\infty}^{\infty}g_j(x)^q\Big)^{\frac{p}{q}}\,d\mu\right)^{\frac{1}{p}}\leq \Big(\sum_{j=-\infty}^{\infty}g_j(x)^q\Big)^{\frac{1}{q}}\leq 2^{\widetilde{k}_0}
\end{equation}
for $x\in E_{\widetilde{k}_0}.$ At the same time the inequality on the left hand side of \eqref{conv} together with \eqref{Chebyschev} imply that
\begin{equation}\label{second}
\frac{\mu(2B_0)}{2}<\mu(2B_0\setminus E_{\widetilde{k}_0-1})\leq 2^{-(\widetilde{k}_0-1)p}\int_{2B_0}\Big(\sum_{j}g_j^q\Big)^{\frac{p}{q}}\,d\mu.
\end{equation}
Combining the inequalities \eqref{first} and \eqref{second} we obtain
\begin{equation}\label{combine}
2^{-(1+\frac{1}{p})}\left(\dashint_{2B_0}\Big(\sum_{j=-\infty}^{\infty}g_j(x)^q\Big)^{\frac{p}{q}}\,d\mu\right)^{\frac{1}{p}}\leq 2^{\widetilde{k}_0}\leq 2^{(1+\frac{1}{p})}\left(\dashint_{2B_0}\left(\sum_{j=-\infty}^{\infty}g_j(x)^q\right)^{\frac{p}{q}}\,d\mu\right)^{\frac{1}{p}}
\end{equation}
Choose the least integer $\ell\in\mathbb{Z}$ such that
\begin{equation}\label{ell}
2^{\ell}>\max\bigg\{2^{1+1/p}\Big(\frac{2}{1-2^{-p/Q}}\Big)^{Q/p}, 1\bigg\}\left(\frac{\mu(2B_0)}{br_0^Q}\right)^{1/p}
\end{equation}
and set $k_0=\widetilde{k}_0+\ell.$ The reason behind such a choice of $\ell$ and $k_0$ will be understood later. Note that $\ell>0,$ by the lower bound of the measure $\mu,$ and hence \eqref{conv} yields $\mu(E_{k_0})>0.$ The inequalities in \eqref{combine} becomes
\begin{equation}\label{final}
2^{k_0}\approx (br_0^Q)^{-1/p}\bigg(\int_{2B_0}\Big(\sum_{j}g_j^q\Big)^{\frac{p}{q}}\,d\mu\bigg)^{1/p}.
\end{equation}
Suppose that $\mu(B_0\setminus E_{k_0})>0$ (we will handle the other case at end of the proof). For $k\geq k_0+1,$ set
\begin{equation}\label{radii}
t_k:=2b^{-1/Q}\mu(2B_0\setminus E_{k-1})^{1/Q}.
\end{equation}
Suppose now that $k\geq k_0+1$ is such that $\mu((E_k\setminus E_{k-1})\cap B_0)>0$ (if such a $k$ does not exist, then $\mu(B_0\setminus E_{k_0})=0,$ contradicting our assumption). Then in particular $t_k>0.$ Pick a point $x_k\in (E_k\setminus E_{k-1})\cap B_0$ and assume that $B(x_k,t_k)\subset 2B_0.$ Then
\begin{equation*}
\mu(B(x_k,t_k))\geq bt_k^Q>\mu(2B_0\setminus E_{k-1})
\end{equation*}
and hence $B(x_k,t_k)\cap E_{k-1}\neq\emptyset.$ Thus there is $x_{k-1}\in E_{k-1}$ such that
\begin{equation*}
d(x_k,x_{k-1})<t_k\leq 2b^{-1/Q}2^{-(k-1)p/Q}\bigg(\int_{2B_0}\Big(\sum_{j}g_j^q\Big)^{\frac{p}{q}}\,d\mu\bigg)^{1/Q},
\end{equation*}
by \eqref{Chebyschev} and \eqref{radii}. Repeating this construction in a similar fashion we obtain for $k\geq k_0+1,$ a sequence of points
\begin{gather*}
x_k \in (E_k\setminus E_{k-1})\cap B_0,\\
x_{k-1} \in E_{k-1}\cap B(x_k,t_k),\\
\vdots  \\
x_{k_0}\in  E_{k_0}\cap B(x_{k_0+1},t_{k_0+1}), 
\end{gather*}
such that
\begin{equation}\label{distance}
d(x_{k-i},x_{k-(i+1)})<t_{k-i}\leq 2b^{-1/Q}2^{-(k-(i+1))p/Q}\bigg(\int_{2B_0}\Big(\sum_{j}g_j^q\Big)^{\frac{p}{q}}\,d\mu\bigg)^{1/Q},
\end{equation}
for every $i=0,1\ldots,k-k_0-1.$ Hence
\begin{eqnarray}\label{totaldistance}
d(x_k,x_{k_0})&<&t_k+t_{k-1}+\cdots +t_{k_0+1}\\ \nonumber
&\leq & 2b^{-1/Q} \bigg(\int_{2B_0}\Big(\sum_{j}g_j^q\Big)^{\frac{p}{q}}\,d\mu\bigg)^{1/Q}\sum_{n=k_0}^{k-1}2^{-np/Q}\\ \nonumber
&=& 2^{-k_0p/Q}\frac{2b^{-1/Q}}{1-2^{-p/Q}}\bigg(\int_{2B_0}\Big(\sum_{j}g_j^q\Big)^{\frac{p}{q}}\,d\mu\bigg)^{1/Q}.
\end{eqnarray}
This is all true provided $B(x_{k-i},t_{k-i})\subset 2B_0$ for $i=0,1,2,\ldots,k-k_0-1.$ That means we require that the right hand side of \eqref{totaldistance} is $\leq r_0\leq\dist(B_0,X\setminus 2B_0).$ Our choice of $k_0,$ \eqref{combine} and \eqref{ell} guarantee us this requirement. Indeed,
\begin{eqnarray*}
2^{k_0}=2^{\widetilde{k}_0+\ell}&\geq & 2^{\ell}2^{-(1+1/p)}\bigg(\dashint_{2B_0}\Big(\sum_{j}g_j^q\Big)^{\frac{p}{q}}\,d\mu\bigg)^{1/p}\\
&\geq & \left(\frac{2}{1-2^{-p/Q}}\right)^{Q/p}(br_0^Q)^{-1/p}\bigg(\int_{2B_0}\Big(\sum_{j}g_j^q\Big)^{\frac{p}{q}}\,d\mu\bigg)^{1/p}.\\
\end{eqnarray*}
Then $t_k+t_{k-1}+\cdots +t_{k_0+1}\leq r_0\leq\dist(B_0,X\setminus 2B_0),$ which implies that $B(x_{k-i},t_{k-i})\subset 2B_0$ for all $i=0,1,2,\ldots,k-k_0-1.$\\
Now we would like to get some upper bound for $\vert u(x_k)\vert$ for $k\geq k_0+1.$ Towards this end, we write
\begin{equation}\label{difference}
\vert u(x_k)\vert \leq\bigg(\sum_{i=0}^{k-k_0-1}\vert u(x_{k-i})-u(x_{k-(i+1)})\vert\bigg)+\vert u(x_{k_0})\vert
\end{equation}
Let us first consider the difference $\vert u(x_{k_0+1})-u(x_{k_0})\vert.$ The inequality \eqref{distance} with $i=k-k_0-1$ gives
\begin{equation}\label{k_0_m_0}
d(x_{k_0+1},x_{k_0})< 2^{m_0-k_0p/Q},
\end{equation}
where $m_0\in\mathbb{Z}$ is such that
\begin{equation}\label{m_0}
2^{m_0-1}\leq 2b^{-1/Q}\bigg(\int_{2B_0}\Big(\sum_{j}g_j^q\Big)^{\frac{p}{q}}\,d\mu\bigg)^{1/Q}<2^{m_0}.
\end{equation}
Now using \eqref{Hajlasz} and \eqref{k_0_m_0}, we get the following bound for the difference
\begin{equation*}
\vert u(x_{k_0+1})-u(x_{k_0})\vert\leq \sum_{j=-\infty}^{m_0-k_0p/Q}2^{js}\Big[g_j(x_{k_0+1})+g_j(x_{k_0})\Big].
\end{equation*}
Similarly, we use the fact that $d(x_{k-i},x_{k-(i+1)})<2^{m_0-(k-(i+1))p/Q},$ and obtain
\begin{equation}
\vert u(x_{k-i})-u(x_{k-(i+1)})\vert\leq \sum_{j=-\infty}^{m_0-(k-i-1)p/Q}2^{js}\Big[g_j(x_{k-i})+g_j(x_{k-(i+1)})\Big],
\end{equation}
for all $i=0,1,2,\ldots,k-k_0-1.$ So, the inequality \eqref{difference} becomes
\begin{equation*}
\vert u(x_k)\vert \leq\bigg(\sum_{i=0}^{k-k_0-1}\sum_{j=-\infty}^{m_0-(k-i-1)p/Q}2^{js}\Big[g_j(x_{k-i})+g_j(x_{k-i-1})\Big]\bigg)+\vert u(x_{k_0})\vert.
\end{equation*}
Use H\"older inequality when $q>1$ and the inequality \eqref{inequality} when $q\leq 1$ and also use the facts that $x_{k-i-1}\in E_{k-i-1}\subset E_{k-i},$ $x_{k-i}\in E_{k-i}$ to obtain
\begin{eqnarray*}
\vert u(x_k)\vert &\leq &\sum_{i=0}^{k-k_0-1}2^{m_0s-\frac{(k-i-1)ps}{Q}}\bigg(\sum_{j=-\infty}^{m_0-(k-i-1)p/Q}\Big[g_j(x_{k-i})+g_j(x_{k-i-1})\Big]^q\bigg)^{1/q}+\vert u(x_{k_0})\vert\\
&\leq &C2^{m_0s}\sum_{i=0}^{k-k_0-1}2^{-\frac{(k-i-1)ps}{Q}}2^{k-i}+\vert u(x_{k_0})\vert.
\end{eqnarray*}
Hence \eqref{distance} with \eqref{m_0}, upon taking supremum over $x_k\in E_k\cap B_0,$ yields
\begin{equation*}
a_k\leq Cb^{-\frac{s}{Q}}\bigg(\int_{2B_0}\Big(\sum_{j}g_j^q\Big)^{\frac{p}{q}}\,d\mu\bigg)^{\frac{s}{Q}}\sum_{n=k_0}^{k-1}2^{n(1-\frac{sp}{Q})}+\sup_{E_{k_0}\cap 2B_0}\vert u\vert.
\end{equation*}
To estimate the last term $\sup_{E_{k_0}\cap 2B_0}\vert u\vert,$ we can assume that $\essinf_{E_{k_0}\cap 2B_0}\vert u\vert=0,$ by the discussion in the beginning of the proof and the fact that $\mu(E_{k_0})>0.$ That means there is a sequence $y_i\in E_{k_0}$ such that $u(y_i)\rightarrow 0$ as $i\rightarrow\infty.$ Therefore, for $x\in E_{k_0}\cap 2B_0$ we have
\begin{equation}\label{lastterm}
\vert u(x)\vert=\lim_{i\rightarrow\infty}\vert u(x)-u(y_i)\vert\leq C'r_0^{s}2^{k_0}.
\end{equation}  
So, for $k>k_0$ we conclude that
\begin{equation}\label{conclusion}
a_k\leq Cb^{-\frac{s}{Q}}\bigg(\int_{2B_0}\Big(\sum_{j}g_j^q\Big)^{\frac{p}{q}}\,d\mu\bigg)^{\frac{s}{Q}}\sum_{n=k_0}^{k-1}2^{n(1-\frac{sp}{Q})}+C'r_0^{s}2^{k_0}.
\end{equation}
For $k\leq k_0,$ we will use the estimate $a_k\leq a_{k_0}\leq C'r_0^s2^{k_0}.$\\
\textbf{Case I:} $0<sp<Q.$ For every $k\in\mathbb{Z},$ we have
\begin{eqnarray*}
a_k &\leq & Cb^{-\frac{s}{Q}}\bigg(\int_{2B_0}\Big(\sum_{j}g_j^q\Big)^{\frac{p}{q}}\,d\mu\bigg)^{\frac{s}{Q}}\sum_{n=-\infty}^k2^{n(1-\frac{sp}{Q})}+C'r_0^{s}2^{k_0}\\
&=& Cb^{-\frac{s}{Q}}2^{k(1-\frac{sp}{Q})}\bigg(\int_{2B_0}\Big(\sum_{j}g_j^q\Big)^{\frac{p}{q}}\,d\mu\bigg)^{\frac{s}{Q}}+C'r_0^{s}2^{k_0}.
\end{eqnarray*}
Applying \eqref{lhs}, \eqref{rhs} and \eqref{final} we get
\begin{align*}
\int_{B_0}\vert u\vert^{p^*}\,d\mu &\leq \sum_{k=-\infty}^{\infty}a_k^{p^*}\mu(B_0\cap (E_k\setminus E_{k-1}))\\
&\leq Cb^{-\frac{sp^*}{Q}}\left(\int_{2B_0}\Big(\sum_{j}g_j^q\Big)^{\frac{p}{q}}\,d\mu\right)^{\frac{sp^*}{Q}}\sum_{k=-\infty}^{\infty}2^{kp}\mu(E_k\setminus E_{k-1})\\
&\qquad+C'r_0^{sp^*}2^{k_0p^*}\mu(B_0)\\
&\leq C\left(1+\frac{\mu(B_0)}{br_0^Q}\right)b^{-\frac{sp^*}{Q}}\left(\int_{2B_0}\Big(\sum_{j}g_j^q\Big)^{\frac{p}{q}}\,d\mu\right)^{p^*/p}.
\end{align*}
Using the fact that $1+\mu(B_0)/br_0^Q\leq 2\mu(B_0)/br_0^Q,$ we get inequality \eqref{embed}.\\
Suppose now that $\mu(B_0\setminus E_{k_0})=0.$ In this case, we use the fact that $\int_{B_0}\vert u\vert^{p^*}\,d\mu=\int_{E_{k_0}}\vert u\vert^{p^*}\,d\mu$ and use inequality \eqref{lastterm} to obtain inequality \eqref{embed}.\\
\textbf{Case II:} $sp=Q.$ It follows from Jensen's inequality that
\begin{equation}\label{simplification}
\left(\dashint_{B_0}\exp\left(C_1b^{1/Q}\frac{\vert u-u_{B_0}\vert}{\Vert \vec{g}\Vert_{L^p(2B_0,l^q)}}\right)\,d\mu\right)^{\frac{1}{2}}\leq \dashint_{B_0}\exp\left(C_1b^{1/Q}\frac{\vert u\vert}{\Vert \vec{g}\Vert_{L^p(2B_0,l^q)}}\right)\,d\mu
\end{equation}
and hence it is enough to estimate the integral on the right hand side of \eqref{simplification}. It follows from \eqref{final} and \eqref{lastterm} that
\begin{equation}\label{zero}
a_{k_0}\leq C'r_0^s2^{k_0}\leq C''b^{-1/p}\bigg(\int_{2B_0}\Big(\sum_{j}g_j^q\Big)^{\frac{p}{q}}\,d\mu\bigg)^{1/p}.
\end{equation}
Hence from \eqref{conclusion} we obtain, for $k>k_0,$
\begin{equation}\label{greaterzero}
a_k\leq \tilde{C}b^{-1/p}\bigg(\int_{2B_0}\Big(\sum_{j}g_j^q\Big)^{\frac{p}{q}}\,d\mu\bigg)^{1/p}(k-k_0).
\end{equation}
We split the integral on the right hand side of \eqref{simplification} into two parts: we estimate the integrals over $B_0\cap E_{k_0}$ and $B_0\setminus E_{k_0}$ separately. For the first part, we have
\begin{eqnarray*}
\frac{1}{\mu(B_0)}\int_{B_0\cap E_{k_0}}\exp\left(C_1b^{1/Q}\frac{\vert u\vert}{\Vert \vec{g}\Vert_{L^p(2B_0,l^q)}}\right)\,d\mu &\leq & \frac{\mu(B_0\cap E_{k_0})}{\mu(B_0)}\exp\left(C_1b^{1/Q}\frac{a_{k_0}}{\Vert \vec{g}\Vert_{L^p(2B_0,l^q)}}\right)\\
&\leq &\exp(C_1C''),
\end{eqnarray*}
where the last inequality follows from \eqref{zero}. The second part is estimated using inequality \eqref{greaterzero} as follows
\begin{eqnarray*}
& &\frac{1}{\mu(B_0)}\int_{B_0\setminus E_{k_0}}\exp\left(C_1b^{1/Q}\frac{\vert u\vert}{\Vert \vec{g}\Vert_{L^p(2B_0,l^q)}}\right)\,d\mu \\
&\leq & \frac{1}{\mu(B_0)}\sum_{k=k_0+1}^{\infty}\exp\left(C_1b^{1/Q}\frac{a_{k_0}}{\Vert \vec{g}\Vert_{L^p(2B_0,l^q)}}\right)\mu(B_0\cap(E_k\setminus E_{k-1}))\\
&\leq & \frac{1}{\mu(B_0)}\sum_{k=k_0+1}^{\infty}\exp\left(C_1\tilde{C}(k-k_0)\right)\mu(E_k\setminus E_{k-1})\\
&\leq & \frac{2^{-k_0Q}}{\mu(B_0)}\sum_{k=-\infty}^{\infty}2^{kQ}\mu(E_k\setminus E_{k-1})\leq C_3,
\end{eqnarray*}
where we have chosen $C_1$ so that $\exp(C_1\tilde{C})=2^Q$ and also we have made use of the inequalities \eqref{rhs}, \eqref{final} and the measure density condition \eqref{measuredensity}.\\
\textbf{Case III:} $sp>Q.$ It follows from \eqref{conclusion} and \eqref{final}, for $k>k_0,$ that
\begin{eqnarray*}
a_k &\leq & Cb^{-\frac{s}{Q}}\bigg(\int_{2B_0}\Big(\sum_{j}g_j^q\Big)^{\frac{p}{q}}\,d\mu\bigg)^{\frac{s}{Q}}\sum_{n=k_0}^{\infty}2^{n(1-\frac{sp}{Q})}+C'r_0^{s}2^{k_0}\\
&\leq & C\left(\frac{\mu(2B_0)}{br_0^Q}\right)^{1/p}r_0^{s}\left(\dashint_{2B_0}\bigg(\sum_{j=-\infty}^{\infty}g_j^q\bigg)^{\frac{p}{q}}\,d\mu\right)^{\frac{1}{p}}.
\end{eqnarray*}
For $k\leq k_0,$ we have 
\begin{equation*}
a_k\leq a_{k_0}\leq C'r_0^s2^{k_0}\leq C\left(\frac{\mu(2B_0)}{br_0^Q}\right)^{1/p}r_0^{s}\left(\dashint_{2B_0}\bigg(\sum_{j=-\infty}^{\infty}g_j^q\bigg)^{\frac{p}{q}}\,d\mu\right)^{\frac{1}{p}}.
\end{equation*}
Therefore
\begin{equation*}
\Vert u-u_{B_0}\Vert_{L^{\infty}(B_0)}\leq 2\Vert u\Vert_{L^{\infty}(B_0)}\leq C\left(\frac{\mu(2B_0)}{br_0^Q}\right)^{1/p}r_0^{s}\left(\dashint_{2B_0}\bigg(\sum_{j=-\infty}^{\infty}g_j^q\bigg)^{\frac{p}{q}}\,d\mu\right)^{\frac{1}{p}}.
\end{equation*}
To prove \eqref{embedc'}, let $x,y\in B_0$ such that  $d(x,y)\leq r_0/4.$ Let us take another ball $B_1=B(x,2d(x,y)).$ Then $2B_1\subset 2B_0$ and hence \eqref{embedc} yields
\begin{equation*}
\vert u(x)-u(y)\vert\leq 2\Vert u-u_{B_1}\Vert_{L^{\infty}(B_1)}\leq Cb^{-1/p}d(x,y)^{1-s/p}\left(\dashint_{2B_0}\bigg(\sum_{j=-\infty}^{\infty}g_j^q\bigg)^{\frac{p}{q}}\,d\mu\right)^{\frac{1}{p}}.
\end{equation*}
If $d(x,y)>r_0/4,$ then upper bound for $\vert u(x)-u(y)\vert$ follows directly from \eqref{embedc} applied on $B_0.$ The proof is complete.
\end{proof}

\section{Haj\l asz-Besov spaces}
From \cite[Theorem 1.73]{Tri06}, we know that, if $p>n/(n+s)$ and $q\leq p^*,$ then $\Vert u\Vert_{L^*(\mathbb{R}^n)}\leq C\Vert u\Vert_{B^s_{p,q}(\mathbb{R}^n)}.$ In the following theorem we have proved embeddings for $N^s_{p,q}(B_0)$ when $0<p<\infty,$ $q\leq p$ and $B_0$ is a fixed ball in a metric space $X.$ 
\begin{theorem}
Let $(X,d,\mu)$ be a metric measure space and $B_0$ be a fixed ball of radius $r_0.$ Let us assume that the measure $\mu$ has a lower bound, that is there exist constants $b, Q>0$ such that $\mu(B(x,r))\geq br^Q$ whenever $B(x,r)\subset 2B_0.$ Let $u\in \dot{N}^s_{p,q}(2B_0)$ and $(g_j)\in \mathbb{D}^s(u)$ where $0<p,q,s<\infty,$ $q\leq p.$ Then there exist constants $C, C_2, C_3$ such that\\
$1.$ If $0<sp<Q,$ then $u\in L^{p^*}(B_0),$ $p^*=\frac{Qp}{Q-sp}$ and
\begin{equation}\label{embed2}
\inf_{c\in\mathbb{R}}\left(\dashint_{B_0}\vert u-c\vert^{p^*}\,d\mu\right)^{\frac{1}{p^*}}\leq C\left(\frac{\mu(2B_0)}{br_0^Q}\right)^{1/p}r_0^{s}\left(\sum_{j=-\infty}^{\infty}\left(\dashint_{2B_0}g_j^p\,d\mu\right)^{\frac{q}{p}}\right)^{\frac{1}{q}}.
\end{equation}
$2.$ If $sp=Q,$ then
\begin{equation}\label{embedb2}
\dashint_{B_0}\exp\left(C_1b^{1/Q}\frac{\vert u-u_{B_0}\vert}{\Vert \vec{g}\Vert_{l^q(2B_0,L^p)}}\right)\,d\mu\leq C_2.
\end{equation}
$3.$ If $sp>Q,$ then
\begin{equation}\label{embedc2}
\Vert u-u_{B_0}\Vert_{L^{\infty}(B_0)}\leq C_3\left(\frac{\mu(2B_0)}{br_0^Q}\right)^{1/p}r_0^{s}\left(\sum_{j=-\infty}^{\infty}\left(\dashint_{2B_0}g_j^p\,d\mu\right)^{\frac{q}{p}}\right)^{\frac{1}{q}}.
\end{equation}
\end{theorem}
\begin{proof}
We would first like to prove the inequality
\begin{equation}\label{embed2'}
\inf_{c\in\mathbb{R}}\left(\dashint_{B_0}\vert u-c\vert^{p^*}\,d\mu\right)^{\frac{1}{p^*}}\leq C\left(\frac{\mu(2B_0)}{br_0^Q}\right)^{1/p}r_0^{s}\left(\sum_{j=-\infty}^{\infty}\dashint_{2B_0}g_j^p\,d\mu\right)^{\frac{1}{p}}.
\end{equation}
Once this is proved, the inequality \eqref{embed2} will immediately follow from the inequality \eqref{inequality}, since $q\leq p.$\\
We may assume by selecting an appropriate constant that $\essinf_E u=0,$ where $E\subset 2B_0$ is any subset of positive measure, since subtracting a constant from $u$ will not affect the inequality \eqref{embed2'}. The set $E$ will be chosen later. With a correct choice of $E$ we will prove \eqref{embed2'} with $(\dashint_{B_0}\vert u\vert ^{p^*}\,d\mu)^{1/p^*}$ on the left hand side.\\
\indent If $g_j=0$ a.e. for all $j,$ then $u$ is constant and hence the theorem follows. Thus we may assume that $\int_{2B_0}g_j^p\,d\mu>0$ for all $j.$ We may also assume that
\begin{equation}\label{lowerbound2}
g_j(x)\geq 2^{-(1+\frac{1}{p})}\left(\dashint_{2B_0}g_j^p\,d\mu\right)^{\frac{1}{p}}>0
\end{equation}
for all $x\in 2B_0$ and all $j\in\mathbb{Z},$ as otherwise we can replace $g_j$ by 
\begin{equation*}
\widetilde{g}_j(x)=g_j(x)+\left(\dashint_{2B_0}g_j^p\,d\mu\right)^{\frac{1}{p}}.
\end{equation*}
Let us define auxiliary sets
\begin{equation*}
E_k=\bigg\{x\in 2B_0:\Big(\sum_{j=-\infty}^{\infty}g_j(x)^p\Big)^{\frac{1}{p}}\leq 2^k\bigg\},\quad k\in\mathbb{Z}.
\end{equation*}
Clearly $E_k\subset E_{k+1}$ for all $k.$ Observe that
\begin{equation}\label{rhs2}
\int_{2B_0}\Big(\sum_j g_j^p\Big)\,d\mu\approx\sum_{k=-\infty}^{\infty}2^{kp}\mu(E_k\setminus E_{k-1}).
\end{equation}
Let $a_k=\sup_{B_0\cap E_k}\vert u\vert.$ Obviously, $a_k\leq a_{k+1}$ and
\begin{equation}\label{lhs2}
\int_{B_0}\vert u\vert^{p^*}\,d\mu\leq\sum_{k=-\infty}^{\infty}a_k^{p^*}\mu(B_0\cap (E_k\setminus E_{k-1})).
\end{equation}
Using Chebyschev's inequality, we get an upper bound for the measure of the complement of $E_{k}$
\begin{eqnarray}\label{Chebyschev2}
\mu(2B_0\setminus E_k) &=& \mu\bigg(\Big\{x\in 2B_0:\Big(\sum_{j=-\infty}^{\infty}g_j(x)^p\Big)^{\frac{1}{p}}>2^k\Big\}\bigg) \nonumber \\
&\leq & 2^{-kp}\int_{2B_0}\Big(\sum_j g_j^p\Big)\,d\mu \nonumber\\
&=& 2^{-kp}\sum_j\Big(\int_{2B_0}g_j^p\,d\mu\Big) \nonumber .
\end{eqnarray}
Lower bound \eqref{lowerbound2} implies that $E_k=\emptyset$ for sufficiently small $k,$ since
\begin{equation}\label{lowerbound2'}
\Big(\sum_{j=-\infty}^{\infty}g_j(x)^p\Big)^{\frac{1}{p}}\geq 2^{-(1+\frac{1}{p})}\bigg(\sum_j\dashint_{2B_0}g_j^p\,d\mu\bigg)^{\frac{1}{p}}>0
\end{equation}
On the other hand $\mu(E_k)\rightarrow \mu(2B_0)$ as $k\rightarrow\infty.$ Hence there is $\widetilde{k}_0\in\mathbb{Z}$ such that
\begin{equation}\label{conv2}
\mu(E_{\widetilde{k}_0-1})<\frac{\mu(2B_0)}{2}\leq\mu(E_{\widetilde{k}_0}).
\end{equation}
The inequality on the right hand side gives $E_{\widetilde{k}_0}\neq\emptyset$ and hence according to \eqref{lowerbound2'}
\begin{equation}\label{first2}
2^{-(1+\frac{1}{p})}\bigg(\sum_j\dashint_{2B_0}g_j^p\,d\mu\bigg)^{\frac{1}{p}}\leq \Big(\sum_{j=-\infty}^{\infty}g_j(x)^p\Big)^{\frac{1}{p}}\leq 2^{\widetilde{k}_0}
\end{equation}
for $x\in E_{\widetilde{k}_0}.$ At the same time the inequality on the left hand side of \eqref{conv2} together with \eqref{Chebyschev2} imply that
\begin{equation}\label{second2}
\frac{\mu(2B_0)}{2}<\mu(2B_0\setminus E_{\widetilde{k}_0-1})\leq 2^{-(\widetilde{k}_0-1)p}\sum_j\int_{2B_0}g_j^p\,d\mu.
\end{equation}
Combining the inequalities \eqref{first2} and \eqref{second2}, we obtain
\begin{equation}\label{combine2}
2^{-(1+\frac{1}{p})}\bigg(\sum_j\dashint_{2B_0}g_j^p\,d\mu\bigg)^{\frac{1}{p}}\leq 2^{\widetilde{k}_0}\leq 2^{(1+\frac{1}{p})}\bigg(\sum_j\dashint_{2B_0}g_j^p\,d\mu\bigg)^{\frac{1}{p}}.
\end{equation}
Choose the least integer $\ell\in\mathbb{Z}$ such that
\begin{equation}\label{ell2}
2^{\ell}>\max\bigg\{2^{1+1/p}\Big(\frac{2}{1-2^{-p/Q}}\Big)^{Q/p}, 1\bigg\}\left(\frac{\mu(2B_0)}{br_0^Q}\right)^{1/p}
\end{equation}
and set $k_0=\widetilde{k}_0+\ell.$ The reason behind such a choice of $\ell$ and $k_0$ will be understood later. Note that $\ell>0,$ by the lower bound of the measure $\mu,$ and hence \eqref{conv2} yields $\mu(E_{k_0})>0.$ The inequalities in \eqref{combine2} becomes
\begin{equation}\label{final2}
2^{k_0}\approx (br_0^Q)^{-1/p}\bigg(\sum_j\int_{2B_0}g_j^p\,d\mu\bigg)^{\frac{1}{p}}.
\end{equation}
Suppose that $\mu(B_0\setminus E_{k_0})>0$ (we will handle the other case at end of the proof). For $k\geq k_0+1,$ set
\begin{equation}\label{radii2}
t_k:=2b^{-1/Q}\mu(2B_0\setminus E_{k-1})^{1/Q}.
\end{equation}
Suppose now that $k\geq k_0+1$ is such that $\mu((E_k\setminus E_{k-1})\cap B_0)>0$ (if such a $k$ does not exist, then $\mu(B_0\setminus E_{k_0})=0,$ contradicting our assumption). Then in particular $t_k>0.$ Pick a point $x_k\in (E_k\setminus E_{k-1})\cap B_0$ and assume that $B(x_k,t_k)\subset 2B_0.$ Then
\begin{equation*}
\mu(B(x_k,t_k))\geq bt_k^Q>\mu(2B_0\setminus E_{k-1})
\end{equation*}
and hence $B(x_k,t_k)\cap E_{k-1}\neq\emptyset.$ Thus there is $x_{k-1}\in E_{k-1}$ such that
\begin{equation*}
d(x_k,x_{k-1})<t_k\leq 2b^{-1/Q}2^{-(k-1)\frac{p}{Q}}\bigg(\sum_j\int_{2B_0}g_j^p\,d\mu\bigg)^{\frac{1}{Q}},
\end{equation*}
by \eqref{Chebyschev2} and \eqref{radii2}. Repeating this construction in a similar fashion we obtain for $k\geq k_0+1,$ a sequence of points
\begin{gather*}
x_k \in (E_k\setminus E_{k-1})\cap B_0,\\
x_{k-1} \in E_{k-1}\cap B(x_k,t_k),\\
\vdots  \\
x_{k_0}\in  E_{k_0}\cap B(x_{k_0+1},t_{k_0+1}), 
\end{gather*}
such that
\begin{equation}\label{distance2}
d(x_{k-i},x_{k-(i+1)})<t_{k-i}\leq 2b^{-1/Q}2^{-(k-(i+1))\frac{1}{Q}}\bigg(\sum_j\Big(\int_{2B_0}g_j^p\,d\mu\bigg)^{\frac{1}{Q}},
\end{equation}
for every $i=0,1\ldots,k-k_0-1.$ Hence
\begin{eqnarray}\label{totaldistance2}
d(x_k,x_{k_0})&<&t_k+t_{k-1}+\cdots +t_{k_0+1}\\ \nonumber
&\leq & 2b^{-1/Q} \bigg(\sum_j\int_{2B_0}g_j^p\,d\mu\bigg)^{\frac{1}{Q}}\\ \nonumber
&=& 2^{-k_0p/Q}\frac{2b^{-1/Q}}{1-2^{-p/Q}}\bigg(\sum_j\int_{2B_0}g_j^p\,d\mu\bigg)^{\frac{1}{Q}}.
\end{eqnarray}
This is all true provided $B(x_{k-i},t_{k-i})\subset 2B_0$ for $i=0,1,2,\ldots,k-k_0-1.$ That means we require that the right hand side of \eqref{totaldistance} is $\leq r_0\leq\dist(B_0,X\setminus 2B_0).$ Our choice of $k_0,$ \eqref{combine2} and \eqref{ell2} guarantee us this requirement.\\
Now we would like to get some upper bound for $\vert u(x_k)\vert$ for $k\geq k_0+1.$ Towards this end, we write
\begin{equation}\label{difference2}
\vert u(x_k)\vert \leq\bigg(\sum_{i=0}^{k-k_0-1}\vert u(x_{k-i})-u(x_{k-(i+1)})\vert\bigg)+\vert u(x_{k_0})\vert
\end{equation}
Let us first consider the difference $\vert u(x_{k_0+1})-u(x_{k_0})\vert.$ The inequality \eqref{distance2} with $i=k-k_0-1$ gives
\begin{equation}\label{k_0_m_02}
d(x_{k_0+1},x_{k_0})< 2^{m_0-k_0p/Q},
\end{equation}
where $m_0\in\mathbb{Z}$ is such that
\begin{equation}\label{m_02}
2^{m_0-1}\leq 2b^{-1/Q}\bigg(\sum_j\int_{2B_0}g_j^p\,d\mu\bigg)^{\frac{1}{Q}}<2^{m_0}.
\end{equation}
Now using \eqref{Hajlasz} and \eqref{k_0_m_02}, we get the following bound for the difference
\begin{equation*}
\vert u(x_{k_0+1})-u(x_{k_0})\vert\leq \sum_{j=-\infty}^{m_0-k_0p/Q}2^{js}\Big[g_j(x_{k_0+1})+g_j(x_{k_0})\Big].
\end{equation*}
Similarly, we use the fact that $d(x_{k-i},x_{k-(i+1)})<2^{m_0-(k-(i+1))p/Q},$ and obtain
\begin{equation}
\vert u(x_{k-i})-u(x_{k-(i+1)})\vert\leq \sum_{j=-\infty}^{m_0-(k-i-1)p/Q}2^{js}\Big[g_j(x_{k-i})+g_j(x_{k-(i+1)})\Big],
\end{equation}
for all $i=0,1,2,\ldots,k-k_0-1.$ So, the inequality \eqref{difference2} becomes
\begin{equation*}
\vert u(x_k)\vert \leq\bigg(\sum_{i=0}^{k-k_0-1}\sum_{j=-\infty}^{m_0-(k-i-1)p/Q}2^{js}\Big[g_j(x_{k-i})+g_j(x_{k-i-1})\Big]\bigg)+\vert u(x_{k_0})\vert.
\end{equation*}
Use H\"older inequality  when $p>1$ and the inequality \eqref{inequality} when $p\leq 1$and also use the facts that $x_{k-i-1}\in E_{k-i-1}\subset E_{k-i},$ $x_{k-i}\in E_{k-i}$ to obtain
\begin{eqnarray*}
\vert u(x_k)\vert &\leq &\sum_{i=0}^{k-k_0-1}2^{m_0s-\frac{(k-i-1)ps}{Q}}\bigg(\sum_{j=-\infty}^{m_0-(k-i-1)p/Q}\Big[g_j(x_{k-i})+g_j(x_{k-i-1})\Big]^p\bigg)^{1/p}+\vert u(x_{k_0})\vert\\
&\leq &C2^{m_0s}\sum_{i=0}^{k-k_0-1}2^{-\frac{(k-i-1)ps}{Q}}2^{k-i}+\vert u(x_{k_0})\vert.
\end{eqnarray*}
Hence \eqref{distance2} with \eqref{m_02}, upon taking supremum over $x_k\in E_k\cap B_0,$ yields
\begin{equation*}
a_k\leq Cb^{-\frac{s}{Q}}\bigg(\sum_j\int_{2B_0}g_j^p\,d\mu\bigg)^{\frac{s}{Q}}\sum_{n=k_0}^{k-1}2^{n(1-\frac{sp}{Q})}+\sup_{E_{k_0}\cap 2B_0}\vert u\vert.
\end{equation*}
To estimate the last term $\sup_{E_{k_0}\cap 2B_0}\vert u\vert,$ we can assume that $\essinf_{E_{k_0}\cap 2B_0}\vert u\vert=0,$ by the discussion in the beginning of the proof and the fact that $\mu(E_{k_0})>0.$ That means there is a sequence $y_i\in E_{k_0}$ such that $u(y_i)\rightarrow 0$ as $i\rightarrow\infty.$ Therefore, for $x\in E_{k_0}\cap 2B_0$ we have
\begin{equation}\label{lastterm2}
\vert u(x)\vert=\lim_{i\rightarrow\infty}\vert u(x)-u(y_i)\vert\leq C'r_0^{s}2^{k_0}.
\end{equation}  
So, for $k>k_0$ we conclude that
\begin{equation}\label{conclusion2}
a_k\leq Cb^{-\frac{s}{Q}}\bigg(\sum_j\int_{2B_0}g_j^p\,d\mu\bigg)^{\frac{s}{Q}}\sum_{n=k_0}^{k-1}2^{n(1-\frac{sp}{Q})}+C'r_0^{s}2^{k_0}.
\end{equation}
For $k\leq k_0,$ we will use the estimate $a_k\leq a_{k_0}\leq C'r_0^s2^{k_0}.$\\
\textbf{Case I:}. For every $k\in\mathbb{Z},$ we have
\begin{equation*}
a_k \leq Cb^{-\frac{s}{Q}}\bigg(\sum_j\int_{2B_0}g_j^p\,d\mu\bigg)^{\frac{s}{Q}}2^{k(1-\frac{sp}{Q})}+C'r_0^{s}2^{k_0}
\end{equation*}
Applying \eqref{lhs2}, \eqref{rhs2} and \eqref{final2} and the measure density condition \eqref{measuredensity} we get
\begin{align*}
\int_{B_0}\vert u\vert^{p^*}\,d\mu &\leq \sum_{k=-\infty}^{\infty}a_k^{p^*}\mu(B_0\cap (E_k\setminus E_{k-1}))\\
&\leq Cb^{-\frac{sp^*}{Q}}\bigg(\sum_j\int_{2B_0}g_j^p\,d\mu\bigg)^{\frac{sp^*}{Q}}\sum_{k=-\infty}^{\infty}2^{kp}\mu(E_k\setminus E_{k-1})\\
&\qquad+C'r_0^{sp^*}2^{k_0p^*}\mu(B_0)\\
&\leq C\left(1+\frac{\mu(B_0)}{br_0^Q}\right)b^{-\frac{sp^*}{Q}}\bigg(\sum_j\int_{2B_0}g_j^p\,d\mu\bigg)^{\frac{p^*}{p}}.\\
\end{align*}
Using the fact that $1+\mu(B_0)/br_0^Q\leq 2\mu(B_0)/br_0^Q,$ we get
\begin{equation*}
\left(\dashint_{B_0}\vert u\vert^{p^*}\,d\mu\right)^{\frac{1}{p^*}}\leq C\left(\frac{\mu(2B_0)}{br_0^Q}\right)^{1/p}r_0^{s}\left(\sum_{j=-\infty}^{\infty}\dashint_{2B_0}g_j^p\,d\mu\right)^{\frac{1}{p}}.
\end{equation*}
Suppose now that $\mu(B_0\setminus E_{k_0})=0.$ In this case, we use the fact that $\int_{B_0}\vert u\vert^{p^*}\,d\mu=\int_{E_{k_0}}\vert u\vert^{p^*}\,d\mu$ and use inequality \eqref{lastterm2} to obtain inequality \eqref{embed2'}. This finishes the proof in this case.\\
\textbf{Case II:}: $sp=Q.$ The proof in this case follows exactly in the same way as the proof of Theorem \ref{embedding} with replacing \eqref{zero} by
\begin{equation}\label{zero2}
a_{k_0}\leq C'r_0^s2^{k_0}\leq C''b^{-1/p}\left(\sum_{j=-\infty}^{\infty}\int_{2B_0}g_j^p\,d\mu\right)^{\frac{1}{p}}
\end{equation}
 and replacing \eqref{greaterzero} by
\begin{equation}\label{greaterzero2}
a_k\leq \tilde{C}b^{-1/p}\left(\sum_{j=-\infty}^{\infty}\int_{2B_0}g_j^p\,d\mu\right)^{\frac{1}{p}}(k-k_0)
\end{equation}
and also using inequality \eqref{inequality} as we have $q\leq p.$\\
The case when $sp>Q$ also follows in a similar fashion. This completes the proof.
\end{proof}
We do not know if one can get the same result as above for $q\leq p^*,$ at least for the non-homogeneous space and when $p>Q/(Q+s).$ In the next theorem, we have relaxed the assumption on $q$ and still have been able to find the same result but with an exponent $p'$ slightly smaller than $p^*$ and this result seems to be new even in $\mathbb{R}^n.$
\begin{theorem}
Let $(X,d,\mu)$ be a metric measure space and $B_0$ be a fixed ball of radius $r_0$ with $2^{l-1}\leq r_0<2^l$ for some integer $l.$ Let us assume that the measure $\mu$ has a lower bound, that is there exist constants $b, Q>0$ such that $\mu(B(x,r))\geq br^Q$ whenever $B(x,r)\subset 2B_0.$ Let $u\in \dot{N}^s_{p,q}(2B_0)$ and $(g_j)\in \mathbb{D}^s(u)$ where $0<p,q,s<\infty.$ Then there exist constants $C,\,C_1,\,C_2$ and $C_3$ such that\\
If $0<sp<Q,$ then $u\in L^{p'}(B_0),$ $p'=\frac{Qp}{Q-(s-s')p}$ for any $0<s'<s.$ Moreover, we have the following inequality:
\begin{equation}\label{embed3}
\inf_{c\in\mathbb{R}}\left(\dashint_{B_0}\vert u-c\vert^{p'}\,d\mu\right)^{\frac{1}{p'}}\leq C\left(\frac{\mu(2B_0)}{br_0^Q}\right)^{1/p}r_0^{s-s'}M,
\end{equation}
where $M:=\left(\sum_{j=-\infty}^{l-2}2^{s'jp}\left(\dashint_{2B_0}g_j^p\,d\mu\right)^{\frac{q}{p}}\right)^{\frac{1}{p}}.$\\
If $sp=Q,$ then
\begin{equation}\label{embed3b}
\dashint_{B_0}\exp\bigg(C_1b^{1/Q}\frac{\vert u-u_{B_0}\vert}{M}\bigg)\,d\mu\leq C_2.
\end{equation}
If $sp>Q,$ then
\begin{equation}\label{embed3c}
\Vert u-u_{B_0}\Vert_{L^{\infty}(B_0)}\leq C_3\left(\frac{\mu(2B_0)}{br_0^Q}\right)^{1/p}r_0^{s-s'}M.
\end{equation}
\end{theorem}
\begin{proof}
First, we would like to establish the following inequality:
\begin{equation}\label{embed3'}
\inf_{c\in\mathbb{R}}\left(\dashint_{B_0}\vert u-c\vert^{p'}\,d\mu\right)^{\frac{1}{p'}}\leq C\left(\frac{\mu(2B_0)}{br_0^Q}\right)^{1/q}r_0^{s-s'}\left(\dashint_{2B_0}\Big(\sum_{j\leq l-2}2^{s'jp}g_j^p\Big)\,d\mu\right)^{\frac{1}{p}}.
\end{equation}
Once this is proved, one can interchange the summation and integration and use the H\"older inequality or inequality \eqref{inequality} to prove \eqref{embed3}.\\
Note that it is enough to prove \eqref{embed3'} with $(\dashint_{B_0}\vert u\vert ^{p'}\,d\mu)^{1/p'}$ on the left hand side. If we have $\sum_{j=-\infty}^{l-2}2^{s'jp}g_j^p=0$ a.e., then $g_j^p=2^{-s'jp}$ a.e. for all $j\leq l-2,$ and hence the theorem follows trivially. Thus we may assume that $\int_{2B_0}\sum_{j\leq l-2}2^{s'jp}g_j^p\,d\mu>0$. We may also assume that
\begin{equation}\label{lowerbound3}
\sum_{j=-\infty}^{l-2}2^{s'jp}g_j(x)^p\geq \frac{1}{2}\left(\dashint_{2B_0}\bigg(\sum_{j=-\infty}^{l-2}2^{s'jp}g_j(x)^p\bigg)\,d\mu\right)>0
\end{equation}
for all $x\in 2B_0.$ Let us define auxiliary sets
\begin{equation*}
E_k=\bigg\{x\in 2B_0:\Big(\sum_{j=-\infty}^{l-2}2^{s'jp}g_j(x)^p\Big)^{\frac{1}{p}}\leq 2^k\bigg\},\quad k\in\mathbb{Z}.
\end{equation*}
Clearly $E_k\subset E_{k+1}$ for all $k.$ Observe that
\begin{equation}\label{rhs3}
\int_{2B_0}\Big(\sum_{j\leq l-2} 2^{s'jp}g_j^p\Big)\,d\mu\approx\sum_{k=-\infty}^{\infty}2^{kp}\mu(E_k\setminus E_{k-1}).
\end{equation}
Let $a_k=\sup_{B_0\cap E_k}\vert u\vert.$ Obviously, $a_k\leq a_{k+1}$ and
\begin{equation}\label{lhs3}
\int_{B_0}\vert u\vert^{p'}\,d\mu\leq\sum_{k=-\infty}^{\infty}a_k^{p'}\mu(B_0\cap (E_k\setminus E_{k-1})).
\end{equation}
By Chebyschev's inequality, we get an upper bound of the complement of $E_k$
\begin{eqnarray}\label{Chebyschev3}
\mu(2B_0\setminus E_k) 
\leq 2^{-kp}\int_{2B_0}\Big(\sum_{j\leq l-2} 2^{s'jp}g_j^p\Big)\,d\mu.
\end{eqnarray}
Lower bound \eqref{lowerbound3} implies that $E_k=\emptyset$ for sufficiently small $k.$ On the other hand $\mu(E_k)\rightarrow \mu(2B_0)$ as $k\rightarrow\infty.$ Hence there is $\widetilde{k}_0\in\mathbb{Z}$ such that
\begin{equation}\label{conv3}
\mu(E_{\widetilde{k}_0-1})<\frac{\mu(2B_0)}{2}\leq\mu(E_{\widetilde{k}_0}).
\end{equation}
The inequality on the right hand side gives $E_{\widetilde{k}_0}\neq\emptyset$ and hence according to \eqref{lowerbound3}
\begin{equation}\label{first3}
2^{-\frac{1}{p}}\left(\dashint_{2B_0}\sum_{j=-\infty}^{l-2}2^{s'jp}g_j(x)^p\,d\mu\right)^{\frac{1}{p}}\leq \Big(\sum_{j=-\infty}^{l-2}2^{s'jp}g_j(x)^p\Big)^{\frac{1}{p}}\leq 2^{\widetilde{k}_0}
\end{equation}
for $x\in E_{\widetilde{k}_0}.$ At the same time the inequality on the left hand side of \eqref{conv3} together with \eqref{Chebyschev3} imply that
\begin{equation}\label{second3}
\frac{\mu(2B_0)}{2}<\mu(2B_0\setminus E_{\widetilde{k}_0-1})\leq 2^{-(\widetilde{k}_0-1)p}\int_{2B_0}\Big(\sum_{j\leq l-2} 2^{s'jp}g_j^p\Big)\,d\mu.
\end{equation}
Combining the inequalities \eqref{first} and \eqref{second} we obtain
\begin{equation}\label{combine3}
2^{-(1+\frac{1}{p})}\left(\dashint_{2B_0}\Big(\sum_{j\leq l-2} 2^{s'jp}g_j^p\Big)\,d\mu\right)^{\frac{1}{p}}\leq 2^{\widetilde{k}_0}\leq 2^{(1+\frac{1}{p})}\left(\dashint_{2B_0}\Big(\sum_{j\leq l-2} 2^{s'jp}g_j^p\Big)\,d\mu\right)^{\frac{1}{p}}.
\end{equation}
Choose the least integer $\ell\in\mathbb{Z}$ such that
\begin{equation}\label{ell3}
2^{\ell}>\max\bigg\{2^{1+1/p}\Big(\frac{2}{1-2^{-p/Q}}\Big)^{Q/p}, 1\bigg\}\left(\frac{\mu(2B_0)}{br_0^Q}\right)^{1/p}
\end{equation}
and set $k_0=\widetilde{k}_0+\ell.$ The reason behind such a choice of $\ell$ and $k_0$ will be understood later. Note that $\ell>0,$ by the lower bound of the measure $\mu,$ and hence \eqref{conv3} yields $\mu(E_{k_0})>0.$ The inequalities in \eqref{combine3} becomes
\begin{eqnarray}\label{final3}
2^{k_0} &\approx & (br_0^Q)^{-1/p}\left(\int_{2B_0}\Big(\sum_{j\leq l-2} 2^{s'jp}g_j^p\Big)\,d\mu\right)^{\frac{1}{p}}\\
&=& (br_0^Q)^{-1/p}\left(\sum_{j\leq l-2} 2^{s'jp}\int_{2B_0}g_j^p\,d\mu\right)^{\frac{1}{p}}.\label{final32}
\end{eqnarray}
Suppose that $\mu(B_0\setminus E_{k_0})>0$ (we will handle the other case at end of the proof). For $k\geq k_0+1,$ set
\begin{equation}\label{radii3}
t_k:=2b^{-1/Q}\mu(2B_0\setminus E_{k-1})^{1/Q}.
\end{equation}
Suppose now that $k\geq k_0+1$ is such that $\mu((E_k\setminus E_{k-1})\cap B_0)>0$ (if such a $k$ does not exist, then $\mu(B_0\setminus E_{k_0})=0,$ contradicting our assumption). Then in particular $t_k>0.$ Pick a point $x_k\in (E_k\setminus E_{k-1})\cap B_0$ and assume that $B(x_k,t_k)\subset 2B_0.$ Then
\begin{equation*}
\mu(B(x_k,t_k))\geq bt_k^Q>\mu(2B_0\setminus E_{k-1})
\end{equation*}
and hence $B(x_k,t_k)\cap E_{k-1}\neq\emptyset.$ Thus there is $x_{k-1}\in E_{k-1}$ such that
\begin{equation*}
d(x_k,x_{k-1})<t_k\leq 2b^{-1/Q}2^{-(k-1)p/Q}\left(\int_{2B_0}\Big(\sum_{j\leq l-2} 2^{s'jp}g_j^p\Big)\,d\mu\right)^{\frac{1}{Q}},
\end{equation*}
by \eqref{Chebyschev3} and \eqref{radii3}. Repeating this construction in a similar fashion we obtain for $k\geq k_0+1,$ a sequence of points
\begin{gather*}
x_k \in (E_k\setminus E_{k-1})\cap B_0,\\
x_{k-1} \in E_{k-1}\cap B(x_k,t_k),\\
\vdots  \\
x_{k_0}\in  E_{k_0}\cap B(x_{k_0+1},t_{k_0+1}), 
\end{gather*}
such that
\begin{equation}\label{distance3}
d(x_{k-i},x_{k-(i+1)})<t_{k-i}\leq 2b^{-1/Q}2^{-(k-(i+1))p/Q}\left(\int_{2B_0}\Big(\sum_{j\leq l-2} 2^{s'jp}g_j^p\Big)\,d\mu\right)^{\frac{1}{Q}},
\end{equation}
for every $i=0,1\ldots,k-k_0-1.$ Hence
\begin{eqnarray}\label{totaldistance3}
d(x_k,x_{k_0})&<&t_k+t_{k-1}+\cdots +t_{k_0+1}\\ \nonumber
&\leq & 2b^{-1/Q} \left(\int_{2B_0}\Big(\sum_{j\leq l-2} 2^{s'jp}g_j^p\Big)\,d\mu\right)^{\frac{1}{Q}}\sum_{n=k_0}^{k-1}2^{-np/Q}\\ \nonumber
&=& 2^{-k_0p/Q}\frac{2b^{-1/Q}}{1-2^{-p/Q}}\left(\int_{2B_0}\Big(\sum_{j\leq l-2} 2^{s'jp}g_j^p\Big)\,d\mu\right)^{\frac{1}{Q}}.
\end{eqnarray}
This is all true provided $B(x_{k-i},t_{k-i})\subset 2B_0$ for $i=0,1,2,\ldots,k-k_0-1.$ That means we require that the right hand side of \eqref{totaldistance} is $\leq r_0\leq\dist(B_0,X\setminus 2B_0).$ Our choice of $k_0,$ \eqref{combine3} and \eqref{ell3} guarantee us this requirement. Indeed,
\begin{eqnarray*}
2^{k_0}=2^{\widetilde{k}_0+\ell}&\geq & 2^{\ell}2^{-(1+1/p)}\bigg(\dashint_{2B_0}\Big(\sum_{j\leq l-2}2^{s'jp}g_j^p\Big)\,d\mu\bigg)^{1/p}\\
&\geq & \left(\frac{2}{1-2^{-p/Q}}\right)^{Q/p}(br_0^Q)^{-1/p}\bigg(\int_{2B_0}\Big(\sum_{j\leq l-2}2^{s'jp}g_j^p\Big)\,d\mu\bigg)^{1/p}.\\
\end{eqnarray*}
Then $t_k+t_{k-1}+\cdots +t_{k_0+1}\leq r_0\leq\dist(B_0,X\setminus 2B_0),$ which implies that $B(x_{k-i},t_{k-i})\subset 2B_0$ for all $i=0,1,2,\ldots,k-k_0-1.$\\
Now we would like to get some upper bound for $\vert u(x_k)\vert$ for $k\geq k_0+1.$ Towards this end, we write
\begin{equation}\label{difference3}
\vert u(x_k)\vert \leq\bigg(\sum_{i=0}^{k-k_0-1}\vert u(x_{k-i})-u(x_{k-(i+1)})\vert\bigg)+\vert u(x_{k_0})\vert
\end{equation}
Let us first consider the difference $\vert u(x_{k_0+1})-u(x_{k_0})\vert.$ The inequality \eqref{distance3} with $i=k-k_0-1$ gives
\begin{equation}\label{k_0_m_03}
d(x_{k_0+1},x_{k_0})< 2^{m_0-k_0p/Q},
\end{equation}
where $m_0\in\mathbb{Z}$ is such that
\begin{equation}\label{m_03}
2^{m_0-1}\leq 2b^{-1/Q}\left(\int_{2B_0}\Big(\sum_{j\leq l-2} 2^{s'jp}g_j^p\Big)\,d\mu\right)^{\frac{1}{Q}}<2^{m_0}.
\end{equation}
Now using \eqref{Hajlasz} and \eqref{k_0_m_03}, we get the following bound for the difference
\begin{equation*}
\vert u(x_{k_0+1})-u(x_{k_0})\vert\leq \sum_{j=-\infty}^{m_0-k_0p/Q}2^{js}\Big[g_j(x_{k_0+1})+g_j(x_{k_0})\Big].
\end{equation*}
Similarly, we use the fact that $d(x_{k-i},x_{k-(i+1)})<2^{m_0-(k-(i+1))p/Q},$ and obtain
\begin{equation}
\vert u(x_{k-i})-u(x_{k-(i+1)})\vert\leq \sum_{j=-\infty}^{m_0-(k-i-1)p/Q}2^{js}\Big[g_j(x_{k-i})+g_j(x_{k-(i+1)})\Big],
\end{equation}
for all $i=0,1,2,\ldots,k-k_0-1.$ So, the inequality \eqref{difference3} becomes
\begin{equation*}
\vert u(x_k)\vert \leq\bigg(\sum_{i=0}^{k-k_0-1}\sum_{j=-\infty}^{m_0-(k-i-1)p/Q}2^{js}\Big[g_j(x_{k-i})+g_j(x_{k-i-1})\Big]\bigg)+\vert u(x_{k_0})\vert.
\end{equation*}
Use H\"older inequality  when $p>1$ and the inequality \eqref{inequality} when $p\leq 1$ and also use the facts that $x_{k-i-1}\in E_{k-i-1}\subset E_{k-i},$ $x_{k-i}\in E_{k-i}$ to obtain
\begin{align*}
\begin{split}
\vert u(x_k)\vert &\leq \sum_{i=0}^{k-k_0-1} 2^{m_0(s-s')-\frac{(k-i-1)p(s-s')}{Q}}\bigg(\sum_{j=-\infty}^{m_0-(k-i-1)p/Q}2^{js'p}\Big[g_j(x_{k-i})+g_j(x_{k-i-1})\Big]^p\bigg)^{1/p}\\
&\qquad+\vert u(x_{k_0})\vert\\
&\leq C2^{m_0(s-s')}\sum_{i=0}^{k-k_0-1}2^{-\frac{(k-i-1)p(s-s')}{Q}}2^{k-i}+\vert u(x_{k_0})\vert.
\end{split}
\end{align*}
Hence \eqref{distance3} with \eqref{m_03}, upon taking supremum over $x_k\in E_k\cap B_0,$ yields
\begin{equation*}
a_k\leq Cb^{-\frac{s-s'}{Q}}\bigg(\int_{2B_0}\Big(\sum_{j\leq l-2}2^{s'jp}g_j^p\Big)\,d\mu\bigg)^{\frac{s-s'}{Q}}\sum_{n=k_0}^{k-1}2^{n\left(1-(s-s')\frac{p}{Q}\right)}+\sup_{E_{k_0}\cap 2B_0}\vert u\vert.
\end{equation*}
To estimate the last term $\sup_{E_{k_0}\cap 2B_0}\vert u\vert,$ we can assume that $\essinf_{E_{k_0}\cap 2B_0}\vert u\vert=0,$ by the discussion in the beginning of the proof and the fact that $\mu(E_{k_0})>0.$ That means there is a sequence $y_i\in E_{k_0}$ such that $u(y_i)\rightarrow 0$ as $i\rightarrow\infty.$ Therefore, for $x\in E_{k_0}\cap 2B_0$ we have
\begin{equation}\label{lastterm3}
\vert u(x)\vert=\lim_{i\rightarrow\infty}\vert u(x)-u(y_i)\vert\leq C'r_0^{s-s'}2^{k_0}.
\end{equation}  
So, for $k>k_0$ we conclude that
\begin{equation}\label{conclusion3}
a_k\leq Cb^{-\frac{s-s'}{Q}}\bigg(\int_{2B_0}\Big(\sum_{j\leq l-2}2^{s'jp}g_j^p\Big)\,d\mu\bigg)^{\frac{s-s'}{Q}}\sum_{n=k_0}^{k-1}2^{n\left(1-(s-s')\frac{p}{Q}\right)}+C'r_0^{s-s'}2^{k_0}.
\end{equation}
For $k\leq k_0,$ we will use the estimate $a_k\leq a_{k_0}\leq C'r_0^{s-s'}2^{k_0}.$\\
\textbf{Case I:} $0<sp<Q.$ Therefore, for every $k\in\mathbb{Z},$ we have
\begin{equation*}
a_k\leq Cb^{-\frac{s-s'}{Q}}2^{k\left(1-(s-s')\frac{p}{Q}\right)}\bigg(\int_{2B_0}\Big(\sum_{j\leq l-2}2^{s'jp}g_j^p\Big)\,d\mu\bigg)^{\frac{s-s'}{Q}}+C'r_0^{s-s'}2^{k_0}.
\end{equation*}
Applying \eqref{lhs3}, \eqref{rhs3} and \eqref{final3} we get
\begin{align*}
\int_{B_0}\vert u\vert^{p'}\,d\mu &\leq \sum_{k=-\infty}^{\infty}a_k^{p'}\mu(B_0\cap (E_k\setminus E_{k-1}))\\
&\leq Cb^{-\frac{(s-s')p'}{Q}}\bigg(\int_{2B_0}\Big(\sum_{j\leq l-2}2^{s'jp}g_j^p\Big)\,d\mu\bigg)^{\frac{(s-s')p'}{Q}}\sum_{k=-\infty}^{\infty}2^{kp}\mu(E_k\setminus E_{k-1})\\
&\qquad+C'r_0^{(s-s')p'}2^{k_0p'}\mu(B_0)\\
&\leq C\left(1+\frac{\mu(B_0)}{br_0^Q}\right)b^{-\frac{(s-s')p'}{Q}}\bigg(\int_{2B_0}\Big(\sum_{j\leq l-2}2^{s'jp}g_j^p\Big)\,d\mu\bigg)^{\frac{p'}{p}}.
\end{align*}
Using the fact that $1+\mu(B_0)/br_0^Q\leq 2\mu(B_0)/br_0^Q,$ we get inequality \eqref{embed3'}.\\
Suppose now that $\mu(B_0\setminus E_{k_0})=0.$ In this case, we use the fact that $\int_{B_0}\vert u\vert^{p'}\,d\mu=\int_{E_{k_0}}\vert u\vert^{p'}\,d\mu$ and use inequality \eqref{lastterm3} to obtain inequality \eqref{embed3'}.\\
\textbf{Case II:} $sp=Q.$ Similar to the proof of Theorem \ref{embedding}, it is enough to prove, after using Jensen's inequality, the inequality \eqref{embed3b} with $\vert u-u_{B_0}\vert$ replaced by $\vert u\vert$ in the left hand side of it. It follows from \eqref{final32}, \eqref{lastterm3} and H\"older inequality (or the inequality \eqref{inequality}) that
\begin{equation}\label{zero3}
a_{k_0}\leq C'r_0^s2^{k_0}\leq C''b^{-1/p}\left(\sum_{j=-\infty}^{l-2}2^{s'jp}\left(\int_{2B_0}g_j^p\,d\mu\right)^{\frac{q}{p}}\right)^{\frac{1}{p}}.
\end{equation}
Hence from \eqref{conclusion3} and H\"older inequality (or the inequality \eqref{inequality}) we obtain, for $k>k_0,$
\begin{equation}\label{greaterzero3}
a_k\leq \tilde{C}b^{-1/p}\left(\sum_{j=-\infty}^{l-2}2^{s'jp}\left(\int_{2B_0}g_j^p\,d\mu\right)^{\frac{q}{p}}\right)^{\frac{1}{p}}(k-k_0).
\end{equation}
Again we split the integral into two parts: we estimate the integral over $B_0\cap E_{k_0}$ and $B_0\setminus E_{k_0}$ separately. For the first part, we have
\begin{eqnarray*}
\frac{1}{\mu(B_0)}\int_{B_0\cap E_{k_0}}\exp\left(\frac{C_1b^{1/Q}\vert u\vert}{M}\right)\,d\mu &\leq & \frac{\mu(B_0\cap E_{k_0})}{\mu(B_0)}\exp\left(\frac{C_1b^{1/Q}a_{k_0}}{M}\right)\\
&\leq &\exp(C_1C''),
\end{eqnarray*}
where the last inequality follows from \eqref{zero3}. The second part is estimated using inequality \eqref{greaterzero3} as follows
\begin{eqnarray*}
& & \frac{1}{\mu(B_0)}\int_{B_0\setminus E_{k_0}}\exp\left(\frac{C_1b^{1/Q}\vert u\vert}{M}\right)\,d\mu \\
&\leq & \frac{1}{\mu(B_0)}\sum_{k=k_0+1}^{\infty}\exp\left(\frac{C_1b^{1/Q}a_{k_0}}{M}\right)\mu(B_0\cap(E_k\setminus E_{k-1}))\\
&\leq & \frac{1}{\mu(B_0)}\sum_{k=k_0+1}^{\infty}\exp\left(C_1\tilde{C}(k-k_0)\right)\mu(E_k\setminus E_{k-1})\\
&\leq & \frac{2^{-k_0Q}}{\mu(B_0)}\sum_{k=-\infty}^{\infty}2^{kQ}\mu(E_k\setminus E_{k-1})\leq C_3,
\end{eqnarray*}
where we have chosen $C_1$ so that $\exp(C_1\tilde{C})=2^Q$ and also we have made use of the inequalities \eqref{rhs3}, \eqref{final3} and the measure density condition \eqref{measuredensity}.\\
\textbf{Case III:} $sp>Q.$ It follows from \eqref{conclusion3} and \eqref{final3}, for $k>k_0,$ that
\begin{eqnarray*}
a_k &\leq & Cb^{-\frac{s}{Q}}\bigg(\sum_{j=-\infty}^{l-2}2^{s'jp}\left(\int_{2B_0}g_j^p\,d\mu\right)^{\frac{q}{p}}\bigg)^{\frac{s}{Q}}\sum_{n=k_0}^{\infty}2^{n(1-\frac{sp}{Q})}+C'r_0^{s}2^{k_0}\\
&\leq & C\left(\frac{\mu(2B_0)}{br_0^Q}\right)^{1/p}r_0^{s}\left(\sum_{j=-\infty}^{l-2}2^{s'jp}\left(\dashint_{2B_0}g_j^p\,d\mu\right)^{\frac{q}{p}}\right)^{\frac{1}{p}}.
\end{eqnarray*}
For $k\leq k_0,$ we have 
\begin{equation*}
a_k\leq a_{k_0}\leq C'r_0^s2^{k_0}\leq C\left(\frac{\mu(2B_0)}{br_0^Q}\right)^{1/p}r_0^{s}\left(\sum_{j=-\infty}^{l-2}2^{s'jp}\left(\dashint_{2B_0}g_j^p\,d\mu\right)^{\frac{q}{p}}\right)^{\frac{1}{p}}.
\end{equation*}
Therefore
\begin{equation*}
\Vert u-u_{B_0}\Vert_{L^{\infty}(B_0)}\leq 2\Vert u\Vert_{L^{\infty}(B_0)}\leq C\left(\frac{\mu(2B_0)}{br_0^Q}\right)^{1/p}r_0^{s}\left(\sum_{j=-\infty}^{l-2}2^{s'jp}\left(\dashint_{2B_0}g_j^p\,d\mu\right)^{\frac{q}{p}}\right)^{\frac{1}{p}}.
\end{equation*}
\end{proof}
\section{Measure density from embedding}
The next theorem shows that, if the space $X$ is $Q$-regular and geodesic, then the measure density condition of a domain $\Omega$ is a necessary condition for the embeddings of both $M^s_{p,q}(\Omega)$ and $N^s_{p,q}(\Omega).$ The proof is inspired by the proof of Theorem 6.1 of \cite{HIT16}, where the measure density condition was derived from extension domains for these spaces.
\begin{theorem}
Let $X$ be a $Q$-regular, geodesic metric measure space and let $\Omega\subset X$ be a domain. Let $0<s<1,$ $0<p<\infty$ and $0<q\leq\infty.$\\
$(i)$ When $sp<Q,$ if there exists a constant $C$ such that for all $f\in M^s_{p,q}(\Omega),$ we have $\Vert f\Vert_{L^{p^*}(\Omega)}\leq C\Vert f\Vert_{M^s_{p,q}(\Omega)},$ where $p^*=\frac{Qp}{Q-sp},$ then $\Omega$ satisfies \eqref{measuredensity}.\\
$(ii)$ When $sp=Q,$ if there exist constants $C_1, C_2$ such that for all $f\in M^s_{p,q}(\Omega)$ and for all balls $B,$ we have
\begin{equation*}
\int_{B\cap\Omega}\exp\left(C_1\frac{\vert u-u_{B(x,r)}\vert}{\Vert f\Vert_{M^s{p,q}(\Omega)}}\right)\,d\mu\leq C_2\mu(B),
\end{equation*}
then $\Omega$ satisfies \eqref{measuredensity}.\\
$(iii)$ When $sp>Q,$ if there exists a constant $C_3$ such that for all $f\in M^s_{p,q}(\Omega),$ and for every $x,y\in\Omega,$ we have $\vert f(x)-f(y)\vert\leq C_3\Vert f\Vert_{M^s_{p,q}(\Omega)}d(x,y)^{s-Q/p},$ then $\Omega$ satisfies \eqref{measuredensity}.
The claims also hold with $M^s_{p,q}(\Omega)$ replaced by $N^s_{p,q}(\Omega).$
\end{theorem}
\begin{proof}
To show that the measure density condition holds, let $x\in \Omega$ and $0<r\leq 1$ and let $B=B(x,r).$ We may assume that $\Omega\setminus B(x,r)\neq\emptyset,$ otherwise the measure density condition is obviously satisfied. We split the proof into three different cases depending on the size of $sp.$\\
\indent\textbf{Case 1:} $0<sp<Q.$ By the proof of \cite[Proposition 13]{HKT08}, the geodesity of $X$ implies that $\mu(\partial B(x,R))=0$ for every $R>0.$ Hence there exist a unique $0<\tilde{r}<r$ such that
\begin{equation*}
\mu(B(x,\tilde{r})\cap\Omega)=\frac{1}{2}\mu(B(x,r)\cap\Omega).
\end{equation*}
Define $u:\Omega\rightarrow [0,1]$ by
\begin{equation}\label{bump}
u(y)=
 \begin{cases}
  1& \text{if $y\in B(x,\tilde{r})\cap\Omega$},\\
  \frac{r-d(x,y)}{r-\tilde{r}} & \text{if $y\in B(x,r)\setminus B(x,\tilde{r})\cap\Omega$},\\
  0& \text{if $y\in \Omega\setminus B(x,r)$}.
 \end{cases}
\end{equation}
Note that
\begin{equation*}
\Vert u\Vert_{L^{p^*}(\Omega)}\geq\mu(B(x,\tilde{r})\cap\Omega)^{1/p^*}.
\end{equation*}
Since the function $u$ is $1/(r-\tilde{r})$-Lipschitz and $\Vert u\Vert_{\infty}\leq 1,$ by \cite[Corollary 3.12]{HIT16} and the fact that $0<r-\tilde{r}<1,$ we have
\begin{equation}\label{corolarry3.12}
\Vert u\Vert_{M^s_{p,q}(\Omega)}\leq C\mu(B(x,r)\cap\Omega)^{1/p}(r-\tilde{r})^{-s}.
\end{equation}
Since $\Vert u\Vert_{L^{p^*}(\Omega)}\leq \Vert u\Vert_{M^s_{p,q}(\Omega)}$ by our assumption, we further have
\begin{equation*}
\mu(B(x,\tilde{r})\cap\Omega)^{1/p^*}\lesssim \mu(B(x,r)\cap\Omega)^{1/p}(r-\tilde{r})^{-s},
\end{equation*}
which yields $r-\tilde{r}\lesssim\mu(B(x,r)\cap\Omega)^{1/Q}.$ Now let us define a sequence $r_0>r_1>r_2>\cdots>0$ by induction:
$$r_0=r, \qquad\text{and}\qquad r_{j+1}=\tilde{r_j}.$$
 Clearly
\begin{equation}\label{decreasing}
\mu(B(x,r_j)\cap\Omega)=2^{-j}\mu(B(x,r)\cap\Omega).
\end{equation}
Therefore $r_j\rightarrow 0$ as $j\rightarrow\infty,$ and hence
\begin{eqnarray*}
r &=&\sum_{j=0}^{\infty}(r_j-r_{j+1})\\
&\lesssim & \sum_{j=0}^{\infty}2^{-j/Q}\mu(B(x,r)\cap\Omega)^{1/Q}\\
&\leq & \mu(B(x,r)\cap\Omega)^{1/Q}
\end{eqnarray*}
as desired.\\
\indent\textbf{Case 2:} $sp=Q.$ Again for $x\in\Omega$ and $0<r\leq 1,$ we will have $0<\tilde{\tilde{r}}<\tilde{r}<r$ such that
\begin{equation}\label{rtilde}
\mu(B(x,\tilde{\tilde{r}})\cap\Omega)=\frac{1}{2}\mu(B(x,\tilde{r})\cap\Omega)=\frac{1}{4}\mu(B(x,r)\cap\Omega).
\end{equation}
Considering the function $u$ associated to $x,\tilde{r},\tilde{\tilde{r}}$ as in \eqref{bump} and using (the proof of) \cite[Corollary 3.12]{HIT16} we have
\begin{equation*}
\int_{B(x,r)\cap\Omega}\exp\left(C_1\frac{\vert u-u_{B(x,r)}\vert(\tilde{r}-\tilde{\tilde{r}})^s}{\mu(B(x,\tilde{r})\cap\Omega)}\right)\,d\mu\leq C_2r^Q.
\end{equation*}
Since $u=v=1$ on $B(x,\tilde{\tilde{r}})\cap\Omega$ and $u=v=0$ on $(B(x,r)\setminus B(x,\tilde{r}))\cap\Omega,$ we have that $\vert u-u_{B(x,r)}\vert\geq 1/2$ on at least one of the sets $B(x,\tilde{\tilde{r}})\Omega$ and $(B(x,r)\setminus B(x,\tilde{r}))\cap\Omega.$ Since the measure of these two sets are comparable to the measure of $B(x,\tilde{r})\cap\Omega,$ we have
\begin{equation*}
\mu(B(x,\tilde{r})\cap\Omega)\exp\big(C_1(\tilde{r}-\tilde{\tilde{r}})^s\mu(B(x,\tilde{r})\cap\Omega)\big)\leq C_2r^Q,
\end{equation*} 
which can be written in the form
\begin{equation}\label{maininequality}
\tilde{r}-\tilde{\tilde{r}}\leq C_1\mu(B(x,\tilde{r})\cap\Omega)^{1/Q}\left[\log\left(\frac{C_2r^Q}{\mu(B(x,\tilde{r})\cap\Omega)}\right)\right]^{1/s}.
\end{equation} 
Now let us state a lemma from \cite{HKT08b}, which will help us to relax the range of $0<r\leq 1$ to $0<r\leq 10\tilde{r}.$
\begin{lemma}
If the measure density condition \eqref{measuredensity} holds for all $x\in\Omega$ and all $r\leq 1$ such that $r\leq 10\tilde{r},$ where $\tilde{r}$ is defined by \eqref{rtilde}, then \eqref{measuredensity} holds fro all $x\in\Omega$ and all $r\leq 1.$
\end{lemma}
\noindent Now let us define a sequence by setting
$$r_0=r, \qquad \text{and}\qquad r_{j+1}=\tilde{r_j}.$$
Inequality \eqref{maininequality} together with the fact that
\begin{equation*}
\mu(B(x,r_{j+1})\cap\Omega)=2^{-j}\mu(B(x,\tilde{r})\cap\Omega)
\end{equation*}
gives
\begin{eqnarray*}
\tilde{r}=\sum_{j=0}^{\infty}(r_j-r_{j+1}) &\leq & \sum_{j=0}^{\infty}C_1\mu(B(x,r_j)\cap\Omega)^{1/Q}\left[\log\left(\frac{C_2r^Q}{\mu(B(x,r_j)\cap\Omega)}\right)\right]^{1/s}\\
&\leq & C_1\mu(B(x,\tilde{r})\cap\Omega)^{1/Q}\sum_{j\in\mathbb{N}}2^{-j/Q}\left[\log\left(\frac{C_22^jr^Q}{\mu(B(x,\tilde{r})\cap\Omega)}\right)\right]^{1/s}.
\end{eqnarray*} 
The sum on the right-hand side is bounded from above (up to a constant) by
\begin{equation*}
\sum_{j=0}^{\infty}2^{-j/Q}j^{1/s}(\log 2)^{1/s}+\left(\sum_{j=0}^{\infty}2^{-j/Q}\right)\left[\log\left(\frac{C_2r^Q}{\mu(B(x,\tilde{r})\cap\Omega)}\right)\right]^{1/s}.
\end{equation*}
The two sums in the above expression converge to some constants depending on $Q$ and $s$ only and hence we obtain
\begin{equation}\label{sumoftwo}
\tilde{r}\leq C\mu(B(x,\tilde{r})\cap\Omega)^{1/Q}\left[1+\left[\log\left(\frac{C_2r^Q}{\mu(B(x,\tilde{r})\cap\Omega)}\right)\right]^{1/s}\right].
\end{equation}
Let us write $\mu(B(x,\tilde{r})\cap\Omega)=\epsilon\tilde{r}^Q.$ Since
\begin{equation*}
\mu(B(x,r)\cap\Omega)=2\mu(B(x,\tilde{r})\cap\Omega)=2\epsilon\tilde{r}^Q\geq 2\cdot 10^{-Q}\epsilon r^Q,
\end{equation*}
it suffices to prove that $\epsilon$ is bounded from below by some positive constant. Now, from inequality \eqref{sumoftwo}, we have
\begin{equation*}
C\epsilon^{1/Q}(1+\log(C_210^Q\epsilon^{-1}))^{1/s}\geq 1.
\end{equation*}
The expression on the left-hand side converges to $0$ if $\epsilon\rightarrow 0,$ and hence $\epsilon$ must be bounded from below by a positive constant.\\
\indent\textbf{Case 3:} $sp>Q.$ For $x\in\Omega$ and $r\in (0,1],$ take $\tilde{r}\in (0,r/4$ and for such $x,r,\tilde{r}$ set $u$ as in \eqref{bump}. Then for all $y,z\in \Omega,$ by our assumption together with \eqref{corolarry3.12}, we have
\begin{equation*}
\vert u(y)-u(z)\vert\leq C\Vert u\Vert_{M^s_{p,q}(\Omega)}d(y,z)^{s-Q/p}\lesssim \frac{\mu(B(x,r)\cap\Omega)^{1/p}}{r^s}d(x,y)^{s-Q/p}.
\end{equation*}
In particular, let $y\in B(x,r)\cap\Omega$ and $z\in (B(x,r+r/2)\cap\Omega)\setminus B(x,r).$ Then $d(y,x)\leq r/4,$ $r\leq d(z,x)\leq 3r/2$ and hence $r/2\leq d(y,z)\leq 2r.$ Therefore, $\mu(B(x,r)\cap\Omega)\gtrsim r^Q.$ This ends the proof of the theorem.
\end{proof}
\begin{remark}
Note that in the previous theorem, we have restricted $s$ to be strictly less than one. For $s=1,$ we refer to a recent result of G\'orka for Haj\l asz-Sobolev space, \cite{Gor17}.
\end{remark}
\def\bibname{References}
\bibliography{embedding}
\bibliographystyle{alpha}
\end{document}